\documentclass[preprint,12pt]{elsarticle} %preprint
\usepackage{amsmath,amsthm,amsfonts,amssymb,graphicx}
\usepackage{enumerate}
\newtheorem{theorem}{Theorem}[section]
\newtheorem{lemma}{Lemma}[section]
\newtheorem{proposition}{Proposition}[section]
\newtheorem{corollary}{Corollary}[section]
\newtheorem{definition}{Definition}[section]
\newtheorem{remark}{Remark}[section]
\newtheorem{example}{Example}[section]
\newtheorem{problem}{Problem}[section]
 \newcommand{\N}{\mathbb{N}}                                            % Symbol of natural numbers
 \newcommand{\R}{\mathbb{R}}                                            % Symbol of real numbers
 \newcommand{\eps}{\varepsilon}                                         % Symbol of epsilon
 \newcommand{\nor}[1]{\left\Vert#1\right\Vert}                          % Symbol of any norm
 \newcommand{\abs}[1]{\left\vert#1\right\vert}                          % Symbol of absolute value
 \newcommand{\er}{\mathcal{R}}				                            % Symbol of an equivalence relation
 
 \newcommand{\de}[1]{\textit{#1}}                             			% Font-type for new concepts

\journal{XXXX}

\begin{document}

\begin{frontmatter}



\title{Existence and uniqueness to several kinds of differential equations using the Coincidence Theory}

%\author[label1,label2]{<author name>}

\author[label1]{D. Ariza-Ruiz}
\ead{dariza@us.es}

\author[label2]{J. Garcia-Falset}
\ead{garciaf@uv.es (the corresponding author)}

\address[label1]{Departamento de An\'alisis Matem\'atico, Universidad de Sevilla, Apdo. 1160, 41080-Sevilla, Spain}

\address[label2]{Departament d'An\`{a}lisi Matem\`{a}tica, Universitat de Val\`{e}ncia, Dr. Moliner 50, 46100, Burjassot, Val\`{e}ncia, Spain}

\begin{abstract}
The purpose of this article is to study the existence of a coincidence point for two mappings defined on a nonempty set and taking values on a Banach space using the fixed point theory for nonexpansive mappings. Moreover, this type of results will be applied to obtain the existence of solutions for some classes of ordinary differential equations.
\end{abstract}


\end{frontmatter}

% \linenumbers

\section{Introduction}
\label{}
Some nonlinear problems arising from most areas of the applied sciences can be formulated under the mathematical point of view involving the study of solutions of equations of the form
\begin{equation}\label{e11}
{\rm Find} \ u\in X \ {\rm such \ that} \ T(u)=S(u),
\end{equation}
where $X$ is a nonempty set, $Y$ is a Banach space, and $T,S:X\to Y$ are two mappings.
\par\medskip
It is well known that the existence of  a solution to problem (\ref{e11}) is, under appropriate conditions, equivalent to the existence of a fixed point for a certain mapping. In this sense, R. Machuca~\cite{Ma.67} proved a coincidence theorem by using Banach's contraction  principle. This same principle was used by K. Goebel~\cite{Go.68} to obtain a similar result under much weaker assumptions, Goebel's theorem allowed the author to give conditions for existence of solutions of the differential equation $x'(t)=f(t,x(t)).$ Recently, several extensions of the above results due to Machuca and Goebel, as well as some application to the existence of solutions for various types of functional equations, have been obtained using generalizations of Banach's principle, for instance see~\cite{Bu.01, GF-Ml.13, OM.13}. On the other hand, in 1977 Gaines and Mawhin~\cite{Gaines} introduced coincidence degree theory. The main goal for them is to search for the existence of solutions of problem (\ref{e11}) in some bounded and open set $X$ in some Banach space for $T$ being a linear operator and $S$ a nonlinear operator using Leray-Schauder degree theory for condensing mappings (see~\cite{GF-HL-Ml.13, Sirma} to find some sharpening results).
\par\medskip
For more than forty years, the study of the existence of fixed points for  nonexpansive mappings has been an important object of research in Nonlinear Functional Analysis, especially, the existence of fixed points for this kind of mappings defined on a closed convex and bounded subset of a Banach space into itself. This theory  was started in 1965 by Browder \cite{Br.65}, G\"{o}hde \cite{Go.65} and Kirk \cite{K.65}. This study was mainly based upon the geometry of the ambient Banach spaces. Thus, in order to simplify the statement, it is usual to say that a Banach space $X$ has the fixed point property for nonexpanisve mappings (FPP for short) whenever each nonexpansive sefmapping of each nonempty closed convex bounded subset of $X$ has a fixed point. Kirk's result says that  those reflexive Banach spaces with normal structure have the FPP. In particular, uniformly convex Banach spaces have normal structure (see \cite{KS.01} for more information).
Nevertheless, it is an easy task to find fixed point free nonexpansive sefmappings defined on  closed convex and unbounded  domains even when the ambient Banach space enjoys the FPP. To solve this problem,  for instance in \cite{GF-ll.10},  the authors considered several fixed point results for nonexpansive mappings with unbounded domains satisfying additional asymptotic contractive-type conditions in terms of a function $G:X\times X\rightarrow\mathbb{R}$ under the following assumptions:
\begin{itemize}
\item[(a)] $G(\lambda x,y)\leq G(x,y)$ for any $x,y\in X$ and $\lambda>0$;
\item[(b)] there exists $S>0$ such that $0< G(x,x)$ for any $x\in X$with $\|x\|\geq S$;
\item[(c)] $G(x+y,z)\leq G(x,z)+G(y,z)$ for any $x,y,z\in X$;
\item[(d)] for each $y\in X$ there exists $t>0$, such that if $\|x\|\geq t,$ then $|G(y,x)|<G(x,x)$.
\end{itemize}
The use of a function $G$ fulfilling the above assumptions yields the following result:
Let $X$ be a Banach space with the FPP and let $C$ be  a closed convex unbounded subset of $X$, If $T:C\rightarrow C$ is a nonexpansive mapping satisfying that there exists $R>0$ such that for every $x\in C$ with $\|x\|\geq R$ the inequality $G(Tx,x)\leq G(x,x)$ holds, then $T$ has a fixed point in $C$. Recently, it has been showed in \cite{G-O} that, if $C$ is a closed convex and unbounded subset of a Banach space $(X,\|\cdot\|)$ and if $T:C\rightarrow C$ is a nonexpansive mapping, then the following statements are equivalent:
\begin{enumerate}
\item There exist $x_0\in C$, $R>0$ and a function $G$ satisfying conditions (a)--(d) such that $G(Tx-x_0,x-x_0)\leq G(x-x_0,x-x_0)$ for all $x\in C$,
\item There exist $x_0\in C$ and $R>0$ such that $Tx-x_0\neq \lambda(x-x_0)$ for all $\lambda>1$ and for all $x\in C$ with $\|x-x_0\|\geq R.$
\end{enumerate}
This fact is one of the main reasons for which our results will be expressed using Leray-Schauder condition.
\par\medskip
On the other hand, in 1981, D. Alspach \cite{Al.81} found a fixed point free nonexpansive mapping leaving invariant a weakly compact convex subset of $L^1[0,1]$, this example shows that there are Banach spaces without the FPP, \cite[Chapter 2]{KS.01}  collects  together examples of fixed point free nonexpansive mappings in a variety of Banach spaces. Therefore, if we wish to obtain positive fixed point results in Banach spaces without the FPP, it will be necessary to add some additional condition on the mapping. In this sense, it is known \cite[Corollary 3.3]{GF-Om.13} that if $C$ is a bounded closed and convex nonempty subset of a Banach space $X$ and $T:C\rightarrow C$ is a nonexpansive mapping such that $I-T$ is $\varphi$-expansive (see definition below), then $T$ has a unique fixed point.
\par\medskip
In this paper, we obtain several versions of the coincidence problem (\ref{e11}) when the Banach space $Y$ has the FPP and also when it fails to have the FPP by using the fixed point theory for nonexpansive mappings. These results allow us to study the existence (and uniqueness) of solutions for the following classes of differential equations:
\par\medskip
\begin{problem} \label{p1} \rm A three-point boundary value problem of second order:
$$
\left\{
\begin{array}{ll}
x''(t)-g\big(t,x(t),x'(t),x''(t)\big)=0 & \quad\text{for } \ {\rm a.e.} \ 0\leq t\leq 1, \\
& \\
x(0)=0,\qquad x'(1)=\delta x'(\eta), & \\
\end{array}
\right.
$$
where $g:[0,1]\times \R^3\to \R$ is a continuous function, $\delta\neq1$ and $\eta\in(0,1)$.
\par\medskip
The multi-point boundary value problems for differential equations arise from many fields of applied mathematics and physics. This kind of problems for linear second order ordinary differential equations was initiated in 1987 by II'in and Moiseev, and motivated by the work of Bitsadze and Samarski on non-local linear elliptic boundary problems. See~\cite{Li.Yu.02} and the references therein.
\end{problem}
\par\medskip
\begin{problem}\label{p2}\rm  A general differential equation with homogeneous Dirichlet condition:
$$\left\{
\begin{array}{l}
	A(u''(t))-\sin\big(u(t)\big)=g(t),\qquad\mbox{for }t\in[0,1]	\\[1em]
	u(0)=0,\quad u(1)=0,
\end{array}\right.$$
where the fixed function $g\in\mathcal{C}[0,1]$ is called the driving force, and $A:\R\to\R$ is a certain known function.
\par\medskip
This type of equations is motivated by the study of the forced oscillations of finite amplitude of a pendulum in the absence of a damping force, see~\cite[Section~4.7]{Tri.85}.
\end{problem}
\par\medskip
\begin{problem}\label{p3}\rm A Cauchy problem with  nonlocal initial data for fractional differential equations of Caputo type:
$$
\left\{
\begin{array}{l}
^cD^q x(t)=f(t,x(t))\qquad\mbox{in }\R_+, 	\\[1em]
x(0)=x_0+g(x),
\end{array}
\right.
$$
where $f\in\mathcal{C}(\R_+\times\R)$, $0<q<1$, and  $^cD^qx$ is the Caputo fractional derivative.
\par\medskip
Fractional derivatives provide an excellent tool for description of memory and hereditary properties of various materials and processes. This is one of the main advantage of fractional differential equations in comparison with classical integer-order models. A vast collection of real-world problems is drawn form fractional equations of Caputo type, see \cite{AN.09} and the references therein.
\end{problem}
%
%====================================================================================================================================================================================================================================================
%
\section{Notations and Preliminaries}
%
%====================================================================================================================================================================================================================================================
%
As it is usual, we shall denote by $\mathcal{C}(I)$  the set of continuous functions of $I$ into $\R$, where $I$ is a subset of $\R$ or $\R^2$.
We shall use $\mathcal{AC}[0,1]$ to denote the set of all real absolutely continuous functions on $[0,1]$.
We shall also use the usual notation of the Sobolev spaces $W^{1,2}[0,1]$ and $W^{2,2}[0,1]$ defined by
$$W^{1,2}[0,1]:=\Big\{x:[0,1]\to\R\;|\;x\in\mathcal{AC}[0,1] \mbox{ with }x'\in L^2[0,1]\Big\},$$
and
$$W^{2,2}[0,1]:=\Big\{x:[0,1]\to\R\;|\;x,x'\in\mathcal{AC}[0,1] \mbox{ with }x''\in L^2[0,1]\Big\},$$
where $L^2[0,1]$ is the  classical Hilbert space endowed with its  usual norm
$$\nor{x}_2:=\left[\int_0^1u(t)^2ds\right]^{\frac{1}{2}}.$$

It is well known that $W^{2,2}[0,1]\subseteq \mathcal{C}^{1}[0,1]:=\{u:[0,1]\to\R\;:\;u'\in\mathcal{C}[0,1]\}$ with continuous injection.
\par\medskip
Let $D$ be a nonempty subset of a normed space $(Y,\nor{\cdot})$.
A mapping $T:D\to Y$ is said to be \de{nonexpansive} if $\nor{Tx-Ty}\leq \nor{x-y}$ for all $x,y\in D$. Recall that a Banach space $Y$ satisfies the \de{Fixed Point Property} (the FPP, in short) whenever each nonexpansive self-mapping of each nonempty closed convex bounded subset of $Y$ has a fixed point.
\par\medskip
Let $D$ be a nonempty subset of $Y$. A mapping $T:D\rightarrow Y$ is said to have \de{the range condition} if
\begin{equation}\label{eq:Cond.Rango}
 \overline{D}\subseteq\bigcap_{\lambda>0}R(I+\lambda\,T),
\end{equation}
where $R(I+\lambda\,T):=\{y\in Y: \ \exists x\in D: \ y=(I+\lambda T)(x)\}$ and $I$ denotes the identity mapping. The following result can be found in~\cite{Ma}.

\begin{lemma}\label{lemma:Cond.Rango}
Let $C$ be a closed convex subset of a Banach space $Y$ and let $h$ be a nonexpansive mapping of $C$ into itself. Then $I-h$ has the range condition.
\end{lemma}

Recall Goebel's theorem~\cite{Go.68} which will be generalized in Section~\ref{sec:Goebel_Thm}.

\begin{theorem}
Let $X$ be an arbitrary nonempty set and let $(Y,\rho)$ be a metric space. Suppose that $S,T:X\to Y$ are two mappings such that $T(X)$ is complete and $S(X)\subseteq T(X)$. If there exists a constant $k\in [0,1)$ such that $\rho(T(x),T(y))\leq k\rho(T(x),T(y))$ for all $x,y\in X$, then the coincidence problem~\eqref{e11} has a solution.
\end{theorem}

There is a number of generalizations of metric spaces. One such generalization is semi-metric spaces initiated by Fr\'{e}chet~\cite{Fr.06}, Menger~\cite{Me.28}, Chittenden~\cite{Ch.17} and Wilson~\cite{Wi.31}.

\begin{definition}
Let $X$ be a nonempty set. A semi-metric is a nonnegative real function $d:X\times X\rightarrow\mathbb{R}_+$ such that
\begin{itemize}
\item[(a)] $d(x,y)=0$ if, and only if, $x=y$;
\item[(b)] $d(x,y)=d(y,x)$ for all $x,y\in X$.
\end{itemize}
\end{definition}
As expected by $(X,d)$ we denote a nonempty set $X$ equipped with a semi-metric $d$ on $X$ and call it a semi-metric space. Note that every metric space (or, more general, every quasi-metric space~\cite{He.01}) is semi-metric but not conversely.
\par\medskip
Throughout this paper, we shall denote by $\mathcal{F}$ the family of all functions $f:\R_+ \to \R_+$ such that
\begin{itemize}
\item[($P_1$)] $f(r) = 0$ if and only if $r = 0$,
\item[($P_2$)] $f$ is nondecreasing.
\end{itemize}

\begin{definition}\label{Defphiexp}
Let $(X,d)$ and $(Y,\rho)$ be two semi-metric spaces. A mapping $A: X\to Y$ is said to be $\varphi$-expansive if there exists a function $\varphi\in\mathcal{F}$ such that
$$\rho(Ax,Ay) \geq \varphi(d(x,y))
\qquad\text{for all }x,y\in X.$$
\end{definition}

%%%%%%%%%%%%%%%%%%%%%%%%%%%%%%%%%%%%%%%%%%%%%%%%%%%%%%%%%%%%%%%%%%%%%%%%%%%%%%%%%%%%%%%%%%%%%%%%%%%%%%%%%%%%%%%%%%%%%%%%%%%%%%%%%%%%%%%%

Finally, recall Bellman's inequality which will be used in the last part of this work (we refer to~\cite[Chapter~XII]{MiPeFi.91} for a book treatment).

\begin{lemma}
If $x:[0,T]\to\R_+$ is a continuous function, $x_0\in\R$, $k\in L^1_{\mbox{loc}}([0,T],\R_+)$, and
$$
x(t)\leq x_0+\int_0^t k(s)\,x(s)\,ds
$$
for each $t\in[0,T]$, then $x(t)\leq x_0\,\mbox{Exp}\big(\int_0^tk(s)\,ds\big)$ for all $t\in[0,T]$.
\end{lemma}
%
%
%====================================================================================================================================================================================================================================================
%
\section{Coincidence problem assuming the FPP}
%
%====================================================================================================================================================================================================================================================
%
In this section we present a positive result to problem~\eqref{e11} when $Y$ is a Banach space enjoying  the FPP. Later on, we will apply such result to discuss the existence of a solution to Problem \ref{p1}.

\begin{theorem}\label{thm:FPP}
Let $X$ be a nonempty set and let $(Y,\nor{\cdot})$ be a Banach space enjoying the FPP. If $T,S:X\to Y$ are two mappings satisfying:
\begin{enumerate}
\item[($i$)] $T(X)$ is a closed convex subset of $Y$,
\item[($ii$)] $S(X) \subseteq T(X)$ and $\nor{S(x)-S(y)}\leq\nor{T(x)-T(y)}$ for all $x,y\in X$,
\item[($iii$)] there exist $x_0\in X$ such that
$$\nor{T(x)-T(x_0)}\geq R\Rightarrow
S(x)-T(x_0)\neq\lambda(T(x)-T(x_0)) \quad\mbox{for all }\lambda>1,$$
\end{enumerate}
then there exists at least one $x$ in $X$ such that $Tx=Sx$.
\end{theorem}

\begin{proof}
Consider $h:T(X)\to2^{T(X)}$ given by $h(x)=S(T^{-1}(x))$, where $T^{-1}(x):=\big\{y\in X:T(y)=x\big\}$. Notice that the mapping $h$ is single valued. Indeed, if $u,v\in h(x)$   since by definition there exist $y,z\in T^{-1}(x)$ such that $u=S(y)$ and $v=S(z)$. Assumption ($ii$) yields that
$$\nor{u-v}=\nor{S(y)-S(z)}\leq\nor{T(z)-T(y)}=\nor{x-x}=0,$$
which implies that $u=v$. Therefore, $h:T(X)\to T(X)$.
\par\medskip
Now, we are going to see that $h$ is nonexpansive. Indeed, fixed $x,y\in T(X)$, there exist $u,v\in X$ with $u\in T^{-1}(x)$ and $v\in T^{-1}(y)$. Then, from assumption~($ii$),
$$\nor{h(x)-h(y)}\leq \nor{S(u)-S(v)}\leq \nor{T(u)-T(v)}=\nor{x-y}.$$
Let us show that $h$ has an almost fixed point (a.f.p. in short) sequence, that is, there exists a sequence $(y_n)$ in $T(X)\subseteq Y$ such that $\lim_{n\to\infty}(y_n-h(y_n))=0$. Let $y_0=T(x_0)\in T(X)$. By~Lemma~\ref{lemma:Cond.Rango}, we know that $I-h:T(X) \to Y$ satisfies the range condition~\eqref{eq:Cond.Rango}. Then, for each $n\in\N$ there exists $y_n=T(x_n)\in T(X)$ such that
$$y_0=\big(I+n(I-h)\big)(y_n),$$
that is,
\begin{equation}\label{eq:thm:FPP:1}
\frac{n+1}{n}(y_0-y_n)=y_0-h(y_n).
\end{equation}
We claim that  $\{y_n\}_{n\in\N}$ is bounded. Indeed, we shall prove that $\nor{y_n-y_0}\leq R$ for all $n\in\N$. Assume, for a contradiction, that $\nor{y_n-y_0}>R$ for some $n\in\N$, that is, $\nor{T(x_n)-T(x_0)}>R$. Bearing in mind~($iii$), we have that $S(x_n)-T(x_0)\neq\lambda(T(x_n)-T(x_0))$ for all $\lambda>1$. From the definition of $h$, we get that $h(y_n)-y_0\neq\lambda(y_n-y_0)$ for all $\lambda>1$ which contradicts ~\eqref{eq:thm:FPP:1}.
\par\medskip
Since $\{y_n\}_{n\in\N}$ is bounded, by~\eqref{eq:thm:FPP:1}, we have that
$$\lim_{n\to\infty}y_n-h(y_n)=\lim_{n\to\infty}\frac{1}{n}(y_0-y_n)=0.$$
Let $C:=\big\{y\in T(X):\limsup_n\nor{y_n-y}\leq\varrho\big\}$, with $\varrho:=\limsup_n\nor{y_n-y_0}$. Note that $C$ is nonempty because $y_0\in C$. It is easy to check that $C$ is closed and convex. Moreover, $h(C)\subseteq C$. Indeed, for any $y\in T(X)$, from the nonexpansiveness of $h$, we have that
$$\nor{y_n-h(y)}\leq\nor{y_n-h(y_n)}+\nor{h(y_n)-h(y)}\leq\nor{y_n-h(y_n)}+\nor{y_n-y}.$$
Taking upper limits as $n\to\infty$, we obtain that
$$\limsup_{n\to\infty}\nor{y_n-h(y)}\leq\limsup_{n\to\infty}\nor{y_n-y}\leq\varrho,$$
because $\{y_n\}_{n\in\N}$ is an a.f.p. sequence.
\par\medskip
Therefore, $h_{|_C}$ is a nonexpansive self-mapping. Since $Y$ has the FPP, there exists at least one $y^* \in C\subseteq T(X)$ such that $h(y^*)=y^*$.  Consider $x^*\in T^{-1}(y^*)$ then $S(x^*)=h(y^*)=y^*=T(x^*)$.
\end{proof}

In the case that $T$ becomes surjective, we can drop some assumptions of the previous theorem.

\begin{corollary}
Let $X$ be a nonempty set and let $(Y,\nor{\cdot})$ be a Banach space enjoying the FPP. If $T,S:X\to Y$ are two mappings satisfying:
\begin{enumerate}
\item[($i$)] $T$ is surjective  and $\nor{S(x)-S(y)}\leq\nor{T(x)-T(y)}$ for all $x,y\in X$,
\item[($ii$)] there exist $x_0\in X$ such that
$$\nor{T(x)-T(x_0)}\geq R\Rightarrow
S(x)-T(x_0)\neq\lambda(T(x)-T(x_0)) \quad\mbox{for all }\lambda>1,$$
\end{enumerate}
then there exists at least one $x$ in $X$ such that $Tx=Sx$.
\end{corollary}
%
%====================================================================================================================================================================================================================================================
%
\subsection{ An existence principle to Problem \ref{p1}. }
\label{sec:Dif.Eq.G}
%
%====================================================================================================================================================================================================================================================
%

The main aim of this section is to establish an existence principle for the three-point boundary value problem. To be more precise, given $g:[0,1]\times \R^3\to \R$ a continuous function, $\delta\neq1$ and $\eta\in(0,1)$, we shall prove, under proper hypotheses, that the following problem
$$(P)
\left\{
\begin{array}{ll}
x''(t)-g\big(t,x(t),x'(t),x''(t)\big)=0 & \quad\text{for } \ {\rm a.e.} \ 0\leq t\leq 1, \\
& \\
x(0)=0,\qquad x'(1)=\delta x'(\eta), & \\
\end{array}
\right.
$$
has at least one solution in $W^{2,2}[0,1]$.
 \par\medskip
In order to do this, we shall need the following technical results, see~\cite[Lemma~2.2]{Pa.11} and~\cite[Theorem~2.3]{Gu-Tr.99} for their proofs, respectively.

\begin{lemma}\label{lemma:Wirtinger-general}
Let $h:[0,1]\to\R_+$ be Lebesgue integrable on each closed interval contained in $(0,1]$ satisfying the following condition:
\begin{equation}\label{cond:h}
\mbox{There exist $\ell\in\R_+$ such that }\;\int_t^1 h(s)\,ds\leq\frac{\ell}{t}\qquad\mbox{ for all }0<t<1.
\end{equation}
Then for each $x\in W^{1,2}[0,1]$, with $x(0)=0$, we have that
\begin{equation}\label{eq:Wirtinger-general}
	\int_0^1 h(t)\,x(t)^2dt\leq 4\ell\int_0^1x'(t)^2dt.
\end{equation}
\end{lemma}

For the sake of simplicity, for any $\ell$, we denote by $\mathcal{Z}(\ell)$ the set of non-negative functions $h:[0,1]\to\R_+$ that are Lebesgue integrable on each closed interval contained in $(0,1]$ and satisfy
$$\int_t^1 h(s)\,ds\leq\frac{\ell}{t}\qquad\mbox{ for all }0<t<1.$$

On one hand, notice that if $h\in L^2[0,1]$ then $h^2\in\mathcal{Z}(\ell)$ with $\ell\geq\nor{h}_2^2$. However, there exist functions $h:[0,1]\to\R$ with $h^2\in\mathcal{Z}(\ell)$, for some $\ell>0$, such that $h\not\in L^2[0,1]$. For instance, $h(t)=\frac{1}{t}\not\in L^2[0,1]$ but $h(t)^2=\frac{1}{t^2}\in\mathcal{Z}(1)$.
\par\medskip
By the other hand, if $h:[0,1]\to\R_+$ is a bounded measurable function with its boundedness constant $\kappa>0$, then $h\in\mathcal{Z}(\tfrac{\kappa}{4})$.
\begin{remark}
In the case that $h$ is a constant function, inequality~\eqref{eq:Wirtinger-general} is not sharp. Indeed, in this case, we have the well-known Wirtinger inequality~\cite[Theorem~256]{Ha-Li-Po.52}. Let $x\in W^{1,2}[0,1]$ be such that $x(0)=0$. Then
\begin{equation}\label{eq:Wirtinger}
\nor{x}_2\leq\frac{2}{\pi}\,\nor{x'}_2.
\end{equation}
\end{remark}

\begin{lemma}\label{lemma:Wirtinger-derivada}
Let $\delta\neq1$, and $\eta\in(0,1)$ be given. Let $x\in W^{2,2}[0,1]$ be such that $x'(1)=\delta x'(\eta)$. Then
$$\nor{x'}_2\leq C(\delta,\eta)\,\nor{x''}_2,$$
where
\begin{equation*}
\begin{split}
C(\delta,\eta)&=
\left\{
\begin{array}{ll}
\min\Big\{\sqrt{F(\delta,\eta)},\dfrac{2}{\pi}\Big\} & \mbox{ if }\delta\leq0, \\
& \\
\sqrt{F(\delta,\eta)} & \mbox{ if }\delta>0,
\end{array}\right.  \\[1.5em]
F(\delta,\eta)&=\frac{1}{2(\delta-1)^2}\left[\delta^2(1-\eta)^2+(\delta^2-2\delta)\eta^2+1\right].
\end{split}
\end{equation*}
\end{lemma}

Consider the set $X:=\left\{u\in W^{2,2}[0,1]: u(0)=0,u'(1)=\delta u'(\eta)\right\}$ and the Hilbert space $Y=(L^2[0,1],\nor{\cdot}_2)$.
Let $T,S:X\to Y$ be defined by $T(u)(t)=u''(t)$ and $S(u)(t)=g\big(t,u(t),u'(t),u''(t)\big)$. It is easy to see that both mappings  $T$ and $S$ are well-defined. Moreover, the mapping $T:X\to Y$ is  surjective. Indeed, given a function $u\in Y$ if we consider the function
$$
v(t)=\int_0^t(t-s)u(s)\,ds
+\frac{t}{1-\delta}\left[\int_0^\eta\delta u(s)\,ds-\int_0^1u(s)\,ds\right],
$$
clearly,  $v\in X$ and $T(v)=u$, that is, $T:X\rightarrow Y$ is an onto mapping. This fact allows us to guarantee  that  $S(X)\subseteq T(X)$.

The following result will be essential in order to show the existence of solutions for the three-point boundary value problem~($P$).
\begin{proposition}\label{prop:desigualdades}
Let $p:[0,1]\to\R$ be such that $p^2\in\mathcal{Z}(\ell)$ for some $\ell\geq0$, and let $Q,R$ be  two constants. For each $x\in X$,
\begin{equation}\label{prop:eq:1}
\int_0^1p(t)\,\abs{x(t)}\,\abs{x'(t)}\,dt\leq 2\sqrt{\ell}\;C(\delta,\eta)^2\nor{x''}_2^2,
\end{equation}
\begin{equation}\label{prop:eq:2}
\int_0^1\big[p(t)\abs{x(t)}+Q\abs{x'(t)}\big]^2dt\leq\Big[2\sqrt{\ell}+Q\Big]^2\,C(\delta,\eta)^2\nor{x''}_2^2,
\end{equation}
and
\begin{equation}\label{prop:eq:3}
\int_0^1\big[p(t)\abs{x(t)}+Q\abs{x'(t)}+R\abs{x''(t)}\big]^2dt\leq\Lambda^2\nor{x''}_2^2,
\end{equation}
where
\begin{equation}\label{prop:constant:Lambda}
\Lambda:=\big(2\sqrt{\ell}+Q\big)\,C(\delta,\eta)+R
\end{equation}
\end{proposition}

\begin{proof}
It is clear that~\eqref{prop:eq:1} is a consequence of H\"{o}lder's inequality, \eqref{eq:Wirtinger-general} and Lemma~\ref{lemma:Wirtinger-derivada}. Let us prove inequality~\eqref{prop:eq:2}. Using Lemma~\ref{lemma:Wirtinger-general}, inequality~\eqref{prop:eq:1} and Lemma~\ref{lemma:Wirtinger-derivada}, we obtain that
\begin{equation*}
\begin{split}
\int_0^1&\big[p(t)\abs{x(t)}+Q\abs{x'(t)}\big]^2dt\\
&=\int_0^1p(t)^2 x(t)^2dt+Q^2\int_0^1x'(t)^2dt+
2Q\int_0^1p(t) \abs{x(t)}\,\abs{x'(t)}\,dt	\\
&\leq4\ell\nor{x'}_2^2+Q^2\nor{x'}_2^2+4Q\sqrt{\ell}\;C(\delta,\eta)^2\nor{x''}_2^2	\\
&\leq4\ell\,C(\delta,\eta)^2\nor{x''}_2^2+Q^2\,C(\delta,\eta)^2\nor{x''}_2^2+
4Q\sqrt{\ell}\;C(\delta,\eta)^2\nor{x''}_2^2	\\
&=\Big[2\sqrt{\ell}+Q\Big]^2\,C(\delta,\eta)^2\nor{x''}_2^2.
\end{split}
\end{equation*}
In order to show~\eqref{prop:eq:3}, we use the previous inequality and H\"{o}lder's inequality
\begin{equation*}
\begin{split}
\int_0^1\big[p(t)\abs{x(t)}+&Q\abs{x'(t)}+R\abs{x''(t)}\big]^2dt \\
&=\int_0^1\big[p(t)\abs{x(t)}+Q\abs{x'(t)}\big]^2dt+R^2\int_0^1x''(t)^2dt \\
&\quad+2R\int_0^1\big[p(t)\abs{x(t)}+Q\abs{x'(t)}\big]\,\abs{x''(t)}dt	\\
&\leq\Big[2\sqrt{\ell}+Q\Big]^2\,C(\delta,\eta)^2\nor{x''}_2^2+R^2\nor{x''}^2 \\
&\quad+2R\nor{x''}_2\sqrt{\int_0^1\big[p(t)\abs{x(t)}+Q\abs{x'(t)}\big]^2dt}	\\
&\leq\Big[2\sqrt{\ell}+Q\Big]^2\,C(\delta,\eta)^2\nor{x''}_2^2+R^2\nor{x''}^2 \\
&\quad+2R\Big[2\sqrt{\ell}+Q\Big]\,C(\delta,\eta)\nor{x''}_2^2	\\
&=\Big[\big(2\sqrt{\ell}+Q\big)\,C(\delta,\eta)+R\Big]^2\nor{x''}_2^2.
\end{split}
\end{equation*}
\end{proof}

Now, we shall assume the following conditions on the function $g$
\begin{itemize}
\item[($H_1$)] There exist two constants $K_2,K_3\geq0$ and a function $k_1:[0,1]\to\R$, with $k_1^2\in\mathcal{Z}(\ell)$ for some $\ell\geq0$, such that
\begin{equation*}
\begin{split}
\abs{g(t,u_1,u_2,u_3)-g(t,v_1,v_2,v_3)}&\leq k_1(t)\,\abs{u_1-v_1} \\
&\qquad+K_2\,\abs{u_2-v_2}+K_3\,\abs{u_3-v_3},
\end{split}
\end{equation*}
for all $t\in[0,1]$ and $u_i,v_i\in\R$ with $i=1,2,3$. And the function $k_1$ and the constants $K_2$ and $K_3$ satisfy $\big(2\sqrt{\ell}+K_2\big)\,C(\delta,\eta)+K_3\leq1$.
\item[($H_2$)] There exist two functions $a_1,a_4:[0,1]\to\R$ with $a_1^2\in\mathcal{Z}(m)$ for some $m\geq0$ and $a_4\in L^2[0,1]$, and two constants $A_2,A_3\geq0$ such that $\big(2\sqrt{m}+A_2\big)\,C(\delta,\eta)+A_3<1$ and
$$\abs{g(t,u_1,u_2,u_3)}\leq a_1(t)\abs{u_1}+A_2\abs{u_2}+A_3\abs{u_3}+a_4(t)$$
for all $t\in[0,1]$ and $u_i\in\R$ with $i=1,2,3$.
\end{itemize}
\begin{theorem} Equation (P) has a solution in $W^{2,2}[0,1]\subseteq C^{1}[0,1]$ whenever assumptions $(H_1)$ and $(H_2)$ are satisfied.
\end{theorem}
\begin{proof}
Using~($H_1$) and inequality~\eqref{prop:eq:3} in Proposition~\ref{prop:desigualdades}, for any $u,v\in X$ we have that
\begin{equation*}
\begin{split}
&\hspace{
-2em}\nor{S(u)-S(v)}^2_2\\
&=\int_0^1\abs{g(t,u(t),u'(t),u''(t))-g(t,v(t),v'(t),v''(t))}^2\,dt	\\
&\leq \int_0^1 \left[k_1(t)\abs{u(t)-v(t)}+K_2\abs{u'(t)-v'(t)}+K_3\abs{u''(t)-v''(t)}\right]^2dt	\\
&\leq \Big[\big(2\sqrt{\ell}+K_2\big)\,C(\delta,\eta)+K_3\Big]^2\nor{u''-v''}_2^2	\\
&\leq \nor{T(u)-T(v)}_2^2,
\end{split}
\end{equation*}
because $\big(2\sqrt{\ell}+K_2\big)\,C(\delta,\eta)+K_3\leq1$.
\par\medskip
Finally, we have to check that $T$ and $S$ fulfill Leray-Schauder condition. In order to see this, since $T(0)=0$, let us assume that $u$ is a solution of the equation:
\begin{equation}\label{00}
T(u)=\mu S(u),\quad\mbox{for some }\mu\in(0,1)
\end{equation}
The above equality implies that
$$u''(t)=\mu g\big(t,u(t),u'(t),u''(t)\big)\quad\mbox{a.e. in }[0,1].$$
Assumption~$(H_2)$ yields
$$|u''(t)|\leq \mu \big( a_1(t)\abs{u(t)}+A_2\abs{u'(t)}+A_3\abs{u''(t)}+a_4(t)\big)
\quad\mbox{ a.e. in }[0,1],$$
hence,
\begin{equation*}
\begin{split}
\nor{u''}_2^2&\leq \mu^2\int_0^1\Big[a_1(t)\abs{u(t)}+A_2\abs{u'(t)}+A_3\abs{u''(t)}\Big]^2dt+\mu^2\nor{a_4}_2^2 \\
&\quad+2\mu^2\int_0^1\Big[a_1(t)\abs{u(t)}+A_2\abs{u'(t)}+A_3\abs{u''(t)}\Big]\,a_4(t)dt.
\end{split}
\end{equation*}
By~\eqref{prop:eq:3} and H\"{o}lder's inequality, we deduce that
\begin{equation*}
\begin{split}
\nor{u''}_2^2&\leq \mu^2\Big[\big(2\sqrt{m}+A_2\big)\,C(\delta,\eta)+A_3\Big]^2\nor{u''}_2^2+\mu^2\nor{a_4}_2^2 \\
&\quad+2\mu^2\nor{a_4}_2\sqrt{\int_0^1\Big[a_1(t)\abs{u(t)}+A_2\abs{u'(t)}+A_3\abs{u''(t)}\Big]^2dt}	\\
&\leq \mu^2\Big[\big(2\sqrt{m}+A_2\big)\,C(\delta,\eta)+A_3\Big]^2\nor{u''}_2^2+\mu^2\nor{a_4}_2^2 \\
&\quad+2\mu^2\nor{a_4}_2\Big[\big(2\sqrt{m}+A_2\big)\,C(\delta,\eta)+A_3\Big]\nor{u''}_2	\\
&=\mu^2\Big[\big(\big(2\sqrt{m}+A_2\big)\,C(\delta,\eta)+A_3\big)\nor{u''}_2+\nor{a_4}_2\Big]^2.
\end{split}
\end{equation*}
Consequently
$$1\leq\mu^2\left[\big(2\sqrt{m}+A_2\big)\,C(\delta,\eta)+A_3+\frac{\nor{a_4}_2}{\nor{u''}_2}\right]^2.$$
Taking limits as $\nor{u''}_2$ goes to infinity we achieve a contradiction, because $\big(2\sqrt{m}+A_2\big)\,C(\delta,\eta)+A_3<1$. Therefore there exists $R>0$ such that if $u$ is a solution of Eq.(\ref{00}), then $\nor{T(u)}\leq R.$
\par\medskip
Bearing in mind that $Y$ is a Hilbert space and therefore it has the FPP, by Theorem~\ref{thm:FPP} we conclude that problem~($P$) has at least one solution in $W^{2,2}[0,1]$.
\end{proof}

Now we shall give an example of a family of functions $g:[0,1]\times\R^3\to\R$ which satisfy hypotheses~($H_1$) and ($H_2$).
\par\medskip
Let $\alpha:[0,1]\to\R$ be such that $\alpha^2\in\mathcal{Z}(m)$ for some $m\geq0$. Let $f_2,f_3:\R\to\R$ be two lipschitzian functions with Lipschitz constants $L_2$ and $L_3$, respectively. Let $\beta:[0,1]\to\R$ be a function in $L^2[0,1]$. Consider $g:[0,1]\times\R^3\to\R$ defined by
$$g(t,u_1,u_2,u_3)=\alpha(t)\frac{2u_1^2}{1+u_1^2}+f_2(u_2)+f_3(u_3)+\beta(t).$$
If $(\frac{3}{2}\sqrt{3m}+L_2)C(\delta,\eta)+L_3\leq1$ then $g$ satisfies~($H_1$) and~($H_2$). Indeed, for any $t\in[0,1]$ and $u_i,v_i\in\R$, with $i=1,2,3$, we have that
\begin{equation*}
\begin{split}
\abs{g(t,u_1,u_2,u_3)-g(t,v_1,v_2,v_3)}&
\leq\abs{\alpha(t)}\,\abs{\frac{2u_1^2}{1+u_1^2}-\frac{2v_1^2}{1+v_1^2}}\\[.5em]
&\quad+\abs{f_2(u_2)-f_2(v_2)}+\abs{f_3(u_3)-f_3(v_3)}	\\[.5em]
&\leq M\abs{\alpha(t)}\,\abs{u_1-u_2} \\[.5em]
&\quad+L_2\,\abs{u_2-v_2}+L_3\,\abs{u_3-v_3},
\end{split}
\end{equation*}
where $M:=\max_{x\in\R}\abs{f_1'(x)}=\frac{3\sqrt{3}}{4}$ with $f_1(x)=\frac{2x^2}{1+x^2}$. Considering $k_1(t)=M\abs{\alpha(t)}$ and  $K_i=L_i$ for $i=2,3$, we get to~($H_1$) because
$$\int_t^1 k_1(s)^2ds=M^2\int_t^1\alpha(s)^2ds\leq \frac{M^2m}{t}\qquad\mbox{ for all }0<t<1,$$
that is, $k_1^2\in\mathcal{Z}(M^2m)$. Since $\abs{f_1(x)}\leq\abs{x}$ for all $x\in\R$ and $f_2,f_3$ are Lipschitz, we infer that
$$
\abs{g(t,u_1,u_2,u_3)}\leq\abs{\alpha(t)}\,\abs{u_1}+L_2\,\abs{u_2}+L_3\,\abs{u_3}+\abs{\beta(t)+f_2(0)+f_3(0)}.
$$
Taking $a_1(t)=\abs{\alpha(t)}$, $A_i=L_i$ for $i=,2,3$, and $a_4(t)=\abs{\beta(t)+f_2(0)+f_3(0)}$, we have that $g$ satisfies~($H_2$) because by hypothesis $a_1^2=\alpha^2\in\mathcal{Z}(m)$ and 
\begin{equation*}
\begin{split}
1&\geq\left(\frac{3}{2}\sqrt{3m}+L_2\right)\,C(\delta,\eta)+L_3\\
&>\big(2\sqrt{m}+L_2\big)\,C(\delta,\eta)+L_3
 =\big(2\sqrt{m}+A_2\big)\,C(\delta,\eta)+A_3.
\end{split}
\end{equation*}
\begin{example}
The problem
$$
\left\{
\begin{array}{l}
\dfrac{x''(t)^3+2x''(t)}{x''(t)^2+3}=\dfrac{\kappa\,x(t)^2}{t+t\,x(t)^2}
+\log\big(t\sqrt{1+2e^{x'(t)}}\,\big) \quad\text{for a.e. }0\leq t\leq1, \\
 \\
x(0)=0,\qquad 10x'(1)+x'(\tfrac{1}{2})=0, \\
\end{array}
\right.
$$
has at least one solution in $W^{2,2}[0,1]$ whenever $\abs{\kappa}\leq\frac{4\pi-6}{9\sqrt{3}}$. Indeed, in this case $C(\delta,\eta)=C(\tfrac{-1}{10},\tfrac{1}{2})=\frac{2}{\pi}$, $\alpha(t)=\frac{\kappa}{2t}\in\mathcal{Z}(\frac{\kappa^2}{4})$, $\beta(t)=\log(t)\in L^2[0,1]$, $f_2(z)=\log\big(\sqrt{1+2e^z}\,\big)$ and
$$f_3(z)=\frac{z}{z^2+3}$$
are two lipschitzian functions with Lipschitz constants $L_2=\frac{1}{2}$ and $L_3=\frac{1}{3}$, respectively. And, moreover,
\begin{equation*}
\begin{split}
\left(\frac{3}{2}\sqrt{3m}+L_2\right)C(\delta,\eta)+L_3\leq1
&\Longleftrightarrow\left(\frac{3\sqrt{3}}{2}\abs{\kappa}+1\right)\frac{1}{\pi}+\frac{1}{3}\leq1\\
&\Longleftrightarrow\abs{\kappa}\leq\dfrac{4\pi-6}{9\sqrt{3}}.
\end{split}
\end{equation*}
\end{example}
%
%====================================================================================================================================================================================================================================================
%
\section{Coincidence problem without assuming the FPP}
%
%====================================================================================================================================================================================================================================================
%
In this section we establish  several  positive result to the Coincidence Problem~\eqref{e11} when $Y$ does not enjoy, necessarily, the FPP. Later on, we will apply such results to discuss the existence and uniqueness of a solution to Problems \ref{p2} and \ref{p3}.
\begin{theorem}\label{thm:Gen.Theorem}
Let $(X,d)$ be a semi-metric space and let $(Y,\nor{\cdot})$ be a Banach space. Assume that $T,S:X\to Y$ are two mappings satisfying:
\begin{enumerate}
\item[($C_1$)] $T(X)$ is a closed convex subset of $Y$,
\item[($C_2$)] $S(X) \subseteq T(X)$ and $\nor{S(x)-S(y)}\leq\nor{T(x)-T(y)}$ for all $x,y\in X$,
\item[($C_3$)] There exists $f\in\mathcal{F}$ such that $f(\nor{T(x)-T(y)})\leq d(x,y)$ for all $x,y\in X$,
\item[($C_4$)] $T-S$ is $\varphi$-expansive,
\item[($C_5$)] there exist $x_0\in X$ and $R>0$ such that
$$\nor{T(x)-T(x_0)}\geq R\Rightarrow
S(x)-T(x_0)\neq\lambda(T(x)-T(x_0)) \quad\mbox{for all }\lambda>1.$$
\end{enumerate}
Then there exists a unique $x$ in $X$ such that $Tx=Sx$.
\end{theorem}

\begin{proof}
Consider $h:T(X)\to2^{T(X)}$ given by $h(x)=S(T^{-1}(x))$, where $T^{-1}(x):=\big\{y\in X:T(y)=x\big\}$. Notice that, the argument developed in the proof of Theorem~\ref{thm:FPP} allows us to infer that $h$ is  a nonexpansive,  single valued mapping and moreover that, there exists a  bounded a.f.p. sequence $(x_n)$ in $T(X).$

Let us see that $(x_n)$ is a Cauchy sequence. Using~($C_2$) and~($C_3$),
\begin{align*}
	\nor{(I-h)(x_n)-(I-h)(x_m)} & =\nor{(T-S)(u_n)-(T-S)(u_m)}	\\
		 & \geq\varphi(d(u_n,u_m))\geq\varphi\big(f(\nor{x_n-x_m})\big),
\end{align*}
where $u_n\in T^{-1}(x_n)$ and $u_m\in T^{-1}(x_m)$. Then, since $(x_n)$ is an a.f.p. sequence it is clear that given $\varepsilon >0$ there exists $n_0\in\mathbb{N}$ such that if $n, m\geq n_0,$ then $\psi(\nor{x_n-x_m}):=\varphi\circ f(\nor{x_n-x_m})<\varepsilon.$ This means that $(x_n)$ is a Cauchy sequence since $\psi\in \mathcal{F}$.

Since $T(X)$ is a closed subset of the Banach space $Y$ there exists $y^{*}\in T(X)$ such that $x_n\rightarrow y^{*}.$ That is, $h(y^*)=y^*$.  Consider $x^*\in T^{-1}(y^*)$ then
$$S(x^*)=h(y^*)=y^*=T(x^*).$$
Moreover, $T$ and $S$ have a unique coincidence point. Indeed, if there exists $x'$ in $X$ such that $S(x')=T(x')$, then from~($C_4$) we have that
$$\varphi(d(x^*,x'))\leq\nor{(T-S)(x^*)-(T-S)(x')}=0.$$
Since $\varphi\in\mathcal{F}$, we obtain that $x^*=x'$.
\end{proof}

\begin{corollary}
Let $(X,d)$ be a semi-metric space and let $(Y,\nor{\cdot})$ be a Banach space. Assume that $T,S:X\to Y$ satisfy the following properties:
\begin{enumerate}
\item[($C_1$)] $T(X)$ is a closed convex subset of $Y$,
\item[($C_2$)] $S(X) \subset T(X)$ and $\nor{S(x)-S(y)}\leq\nor{T(x)-T(y)}$ for all $x,y\in X$,
\item[($C_3^*$)] $T$ is uniformly continuous,
\item[($C_4$)] $T-S$ is $\varphi$-expansive,
\item[($C_5$)] there exist $x_0\in X$ and $R>0$ such that
$$\nor{T(x)-T(x_0)}\geq R\Rightarrow
S(x)-T(x_0)\neq\lambda(T(x)-T(x_0)) \quad\mbox{for all }\lambda>1.$$
\end{enumerate}
Then there exists a unique $x$ in $X$ such that $T(x)=S(x)$.
\end{corollary}

\begin{proof}
It is enough to see that condition ($C_3$) in Theorem \ref{thm:Gen.Theorem} holds.  To this end, we define the function $f:\R_+\to\R_+$ by
$$f(r):=\!\inf_{\substack{
   x,y\in T(X) \\
   \|x-y\|\geq r
  }}
\hspace{-6pt}\Big\{\ d(u,v):u\in T^{-1}(x),v\in T^{-1}(y)\Big\}.$$
It is easy to check that $f$ is well-defined and $f\in\mathcal{F}$. Let us prove that if $r>0$ then $f(r)>0$. Arguing by contradiction, assume that there exist two sequences $\{x_n\}$ and $\{y_n\}$ such that $\nor{x_n-y_n}\geq r$ and there exists $\{u_n\}$ and $\{v_n\}$ with $u_n\in T^{-1}(x_n)$ and $v_n\in T^{-1}(y_n)$ for each $n\in\N$, such that $d(u_n,v_n)\to 0$ as $n\to+\infty$. Since $T$ is uniformly continuous, we get that
$$r\leq\limsup_{n\to\infty}\nor{x_n-y_n}=\limsup_{n\to\infty}\nor{T(u_n)-T(v_n)}=0,$$
which is a contradiction.
\end{proof}
%
%====================================================================================================================================================================================================================================================
%
\subsection{A generalization of Goebel's Theorem}\label{sec:Goebel_Thm}
%
%====================================================================================================================================================================================================================================================
%
In 1973, Geraghty~\cite{Ge.73} gave an interesting generalization of the contraction principle using the class $\boldsymbol{\mathcal{S}}$ of the functions $\alpha:\R_+\to[0,1)$ satisfying the following condition:
\begin{equation}\label{eq:cond.Geraghty}
\lim_{n\to\infty}\alpha(t_n)=1\quad\text{implies}\quad\lim_{n\to\infty}t_n=0.
\end{equation}

In this section we shall give a generalization of Goebel's Theorem in the setting of Banach spaces.

\begin{theorem}\label{thm:Coincidence.Contraction}
Let $X$ be a nonempty set and let $(Y,\nor{\cdot})$ be a Banach space. If   $T,S:X\to Y$ are two mappings such that  $T$ is onto and there exists a decreasing function $\alpha\in\boldsymbol{\mathcal{S}}$ satisfying
\begin{equation}\label{eq:Contractive.Version}
\nor{Sx-Sy}\leq\alpha\big(\nor{Tx-Ty}\big)\,\nor{Tx-Ty}
\quad\mbox{for all }x,y\in X.
\end{equation}
Then, there exists at least one $x^*\in X$ such that $Tx^*=Sx^*$. If, in addition, $T$ is injective, then the coincidence point $x^*$ is unique.
\end{theorem}
\begin{proof}
Let us consider the equivalence relation $\er$ in $X$ defined by $x\er y$ if, and only if, $Tx=Ty$. Given $x\in X$, we denote $\tilde{x}:=\{y\in X: y\er x\}$ and consider the set
$$\tilde{X}:=\Big\{\,\tilde{x}: \ x\in X\Big\}, $$
on $\tilde{X}$ we can  define the metric $d:\tilde{X}\times\tilde{X}\to\R_+$ given by $d(\tilde{x},\tilde{y})=\nor{Tx-Ty}$ where $x\in\tilde{x}$ and $y\in\tilde{y}$. By definition of $\er$, we have that the mapping $\tilde{T}:\tilde{X}\to Y$ given by $\tilde{T}\tilde{x}=Tx$ is well-defined. By~\eqref{eq:Contractive.Version}, we infer that the mapping $\tilde{S}:\tilde{X}\to Y$ given by $\tilde{S}\tilde{x}=Sx$ is also well-defined. Indeed, if $y_1,y_2\in\tilde{x}$, by definition of $\er$, $Ty_1=Ty_2$. Then, $\nor{Sy_1-Sy_2}\leq\nor{Ty_1-Ty_2}=0$, that is, $Sy_1=Sy_2$.
\par\medskip
Now we will show that $\tilde{T}$ and $\tilde{S}$ satisfy the hypotheses of Theorem \ref{thm:Gen.Theorem}. Since $T$ is onto,  $Y$ is a Banach space and $0\leq\alpha(t)<1$ for all $t\geq0$, clearly  conditions $(C_1)-(C_2)$ are satisfied.  Moreover, by definition of $d$ and $\tilde{T}$, we have  that $\tilde{T}$ is an isometry. Hence, condition $(C_3)$  holds. Furthermore, for any $x,y\in X$,
\begin{equation*}
\begin{split}
\Vert(\tilde{T}-\tilde{S})\tilde{x}-(\tilde{T}-\tilde{S})\tilde{y}\Vert&=\nor{(T-S)x-(T-S)y}\\
	&\geq \nor{Tx-Ty}-\nor{Sx-Sy}	\\
	&\geq \nor{Tx-Ty}-\alpha\big(\nor{Tx-Ty})\,\nor{Tx-Ty}	\\
	&=\varphi(d\big(\tilde{x},\tilde{y}\big)),
\end{split}
\end{equation*}
that is, $\tilde{T}-\tilde{S}$ is $\varphi$-expansive with $\varphi(t)=(1-\alpha(t))t$, because $\alpha$ is decreasing, i.e., condition $(C_4)$ is satisfied.

Finally, let us prove condition~($C_5$). Since $T$ is onto, there exists $x_0\in X$ such that $\tilde{T}\tilde{x_0}=0$. By contradiction, assume that for each $n\in\N$ there exist $\lambda_n>1$ and $x_n\in X$, with $\nor{Tx_n}\geq n$, such that $Sx_n=\lambda_n\,Tx_n$. Then, using~\eqref{eq:Contractive.Version}, we have that $\nor{Sx_n-Sx_0}\leq\alpha\big(\nor{Tx_n}\big)\nor{Tx_n}$. By the triangle inequality, $\lambda_n\nor{Tx_n}\leq\alpha\big(\nor{Tx_n}\big)\nor{Tx_n}+\nor{Sx_0}$, that is,
$$1<\lambda_n\leq\alpha\big(\nor{Tx_n}\big)+\frac{\nor{Sx_0}}{\nor{Tx_n}}\leq\alpha\big(\nor{Tx_n}\big)+\frac{\nor{Sx_0}}{n}.$$
Taking limits as $n\to\infty$ and bearing in mind that $\alpha(t)<1$ for all $t\geq0$, we obtain that
\begin{equation}\label{ger}\lim_{n\to\infty}\alpha\big(\nor{Tx_n}\big)=1.\end{equation}
Now, (\ref{ger}) along with the fact $\alpha\in\boldsymbol{\mathcal{S}}$ imply that  $\nor{Tx_n}\to0$ as $n\to\infty$, which is a contradiction. Therefore, there exists $R>0$ such that if $\nor{Tx}\geq R$, then $Sx\neq\lambda Tx$ for all $\lambda>1$. That is,  $\tilde{T}$ and $\tilde{S}$ satisfy condition~($C_5$).
\par\medskip
  Since $\tilde{T}$ and $\tilde{S}$ fulfill the conditions of Theorem  \ref{thm:Gen.Theorem}, there exists a unique $\tilde{x^*}\in\tilde{X}$ such that $\tilde{T}\tilde{x^*}=\tilde{S}\tilde{x^*}$. Using the definition of $\tilde{T}$ and $\tilde{S}$, we conclude that $Tx^*=Sx^*$.
\par\medskip
Let us assume now that $T$ is injective. In this case, we only have to show that $\tilde{x^{*}}=\{x^*\}$. Note that if there exist two  points $x_1^*, x_2^*\in\tilde{x^*}$, then $T(x_1^*)=T(x_2^*)$ and the injectivity   of  $T$ yields that $x_1^*=x_2^*.$

\end{proof}

As a consequence of the above result we obtain Goebel's theorem in the setting of Banach spaces.

\begin{corollary}\label{cor:Coincidence.Contraction}
Let $A$ be an arbitrary nonempty set, $(N,\nor{\cdot})$ be a normed space and $T,S:A\to N$. Assume that $S(A)\subseteq T(A)$, $T(A)$ is a complete subspace of $N$ and there exists a constant $k\in[0,1)$ such that
\begin{equation*}
\nor{Sx-Sy}\leq k\nor{Tx-Ty}\quad\mbox{for all }x,y\in A.
\end{equation*}
Then, there exists at least one $x^*\in A$ such that $Tx^*=Sx^*$. If, in addition, $T$ is injective, then the coincidence point $x^*$ is unique.
\end{corollary}
\begin{proof}
Just taking  $X=A$ and $Y=T(X)$ with the norm $\nor{\cdot}$. Since the  function $\alpha(t)=\left\{\begin{array}{ll}k, & t>0\\1,& t=0\end{array}\right.$ is decreasing and belongs to $\boldsymbol{\mathcal{S}}$, we deduce the result from Theorem~\ref{thm:Coincidence.Contraction}.
\end{proof}
%
%====================================================================================================================================================================================================================================================
%
\subsection{On the generalized Ulam-Hyers stability}
%
%====================================================================================================================================================================================================================================================
%
The stability problem of functional equations originated from a question of Ulam~\cite{Ul.64}, in 1940, concerning the stability of group homomorphisms. In the following year, Hyers~\cite{Hy.41} gave a first affirmative partial answer to the question of Ulam for Banach spaces. Thereafter, this type of stability is called the Ulam-Hyers stability.

\begin{definition}[\cite{Ru.09}]
Let $(X,d)$ be a semi-metric space and let $(Y,\nor{\cdot})$ be a Banach space. The coincidence problem~\eqref{e11} is called generalized Ulam-Hyers stable if and only if there exists $\psi : \mathbb{R}_+ \to \mathbb{R}_+$ increasing, continuous at $0$ and $\psi(0) = 0$ such that for every $\varepsilon > 0$ and for each solution $w^* \in X$ of the approximative coincidence problem
\begin{equation}\label{appcoinprob}
\nor{Tw^*-Sw^*} \leq \varepsilon
\end{equation}
there exists a solution $z^*$ of problem~\eqref{e11} such that
$$d(w^*,z^*)\leq \psi(\varepsilon).$$
If there exists $c > 0$ such that $\psi(t) = ct$ for each $t\in \mathbb{R}_+$ then the coincidence problem (\ref{e11}) is said to be Ulam-Hyers stable.
\end{definition}

\begin{proposition}\label{prop:Ulam.Hyers.stability}
Suppose that we are under the assumptions of Theorem~\ref{thm:Gen.Theorem} and additionally that the function $\varphi\in\mathcal{F}$ is strictly increasing and onto. Then, the coincidence problem~\eqref{e11} is generalized Ulam-Hyers stable.
\end{proposition}
\begin{proof}
Let $\varepsilon > 0$ and $w^* \in X$ such that
$$\nor{T(w^*) - S(w^*)}\leq \varepsilon.$$
Taking $u \in X$ as the unique solution of the coincidence problem~\eqref{e11} we have
$$\varepsilon > \nor{Tw^*-Sw^*}=\nor{(T-S)w^*-(T-S)u}\geq\varphi\big(d(w^*,u)\big),$$
then $d(w^*,u)\leq \varphi^{-1}(\varepsilon)$.  Which means that  problem (\ref{e11}) is generalized Ulam-Hyers stable.
\end{proof}

\begin{remark}
Notice that if $T-S$ is expansive, that is, there exists a constant $c>0$ such that for all $x,y\in X$
$$\nor{(T-S)x-(T-S)y}\geq c\,d(x,y),$$
then the coincidence problem~\eqref{e11} is Ulam-Hyers stable.
\end{remark}
%
%====================================================================================================================================================================================================================================================
%
\section{Applications}
%
%====================================================================================================================================================================================================================================================
%
In this section we will show two applications of results from the previous section to Differential Equations and Fractional Differential Equations.
%
%====================================================================================================================================================================================================================================================
%
\subsection{An existence and uniqueness result to Problem \ref{p2}}
\label{sec:Dif.Eq.SINE}
%
%====================================================================================================================================================================================================================================================
%
We shall apply our coincidence point result (Theorem~\ref{thm:Gen.Theorem}) in order to study the existence of classical solutions for the following general differential equation with homogeneous Dirichlet condition
$$(P)\left\{
\begin{array}{l}
	A(u''(t))-\sin\big(u(t)\big)=g(t),\qquad\mbox{for }t\in[0,1]	\\[1em]
	u(0)=0,\quad u(1)=0,
\end{array}\right.$$
where the fixed function $g\in \mathcal{C}[0,1]$, and $A:\R\to\R$ is a certain known function.
\par\medskip

Let $Y=(\mathcal{C}[0,1],\nor{\cdot}_\infty)$ be the classical Banach space of the real continuous functions $u:[0,1]\to\R$ endowed with its usual norm $\nor{u}_\infty:=\max\{\abs{u(t)}:t\in[0,1]\}$. In the linear space $\mathcal{C}^2[0,1]:=\{u:[0,1]\to\R\;:\;u''\in\mathcal{C}[0,1]\}$ we introduce the following linear subspace
$$X:=\big\{u\in\mathcal{C}^2[0,1]:u(0)=u(1)=0\big\}.$$
Notice that $X$ endowed with the norm $\nor{u}_*:=\max\big\{\nor{u}_\infty,\nor{u'}_\infty,\nor{u''}_\infty\big\}$ is a Banach space. Using the Mean Value Theorem, one can prove, see \cite{GF-Ml.13}, that $\nor{u}_*=\nor{u''}_\infty$ for all $u\in X$.
\par\medskip

\begin{proposition}\label{pro:Dif.Eq}
With the previous notation, assume that $A$ satisfies the following two properties:
\begin{itemize}
\item[($A_1$)] $A$ is continuous;
\item[($A_2$)] there exists a function $f\in\mathcal{F}$ such that $f(\abs{Ax-Ay})\leq\abs{x-y}\leq\abs{Ax-Ay}$, for all $x,y\in\R$.
\end{itemize}
Then, problem ($P$) has a unique solution in $\mathcal{C}^2[0,1]$ and, moreover, ($P$) is generalized Ulam-Hyers stable.
\end{proposition}

\begin{proof}
We define $T,S:X\to Y$ by
$$
T(u)(t)=A(u''(t))\qquad\mbox{and}\qquad
S(u)(t)=\sin\big(u(t)\big)+g(t).$$

It is immediate to show that $S$ is a nonexpansive mapping and that the Leray-Schauder condition holds because $S$ is bounded. Moreover, $T$ and $S$ satisfy~($C_2$) in Theorem~\ref{thm:Gen.Theorem}. Indeed, for each $u,v\in X$,  the expansiveness of $A$ yields
\begin{equation*}
\begin{split}
	\nor{S(u)-S(v)}_\infty&=\max_{t\in[0,1]}\abs{\sin(u(t))-\sin(v(t))}
		\leq\max_{t\in[0,1]}\abs{u(t)-v(t)}	\\
		&\leq\nor{u''-v''}_\infty
		=\max_{t\in[0,1]}\abs{u''(t)-v''(t)}	\\
		&\leq\max_{t\in[0,1]}\abs{A(u''(t))-A(v''(t))}	\\
		&=\nor{T(u)-T(v)}_\infty.
\end{split}
\end{equation*}

Now we shall prove that $T$ satisfies~($C_3$). Let $u,v\in X$. By~($A_2$), we have that
\begin{equation*}
\begin{split}
f(\abs{T(u)(t)-T(v)(t)})&=f(\abs{A(u''(t))-A(v''(t))})
 	 \leq \abs{u''(t)-v''(t)} \\
	&\leq \abs{A(u''(t))-A(v''(t))}\leq \nor{T(u)-T(v)}_\infty.\\
\end{split}
\end{equation*}
Then,
$$\max_{0\leq t\leq 1}f(\abs{T(u)(t)-T(v)(t)})\leq \nor{u-v}_*\leq \nor{T(u)-T(v)}_\infty.$$
Since $f$ is nondecreasing, we obtain that
\begin{equation}\label{eq:Dif.Eq:01}
f(\nor{T(u)-T(v)}_\infty)\leq \nor{u-v}_*\leq \nor{T(u)-T(v)}_\infty.
\end{equation}
Let us prove that $T-S$ is $\varphi$-expansive, where $\varphi:\R^+\to\R^+$, given by
$$\varphi(r):=\left\{
\begin{array}{ll}
	r-2\sin\big(\tfrac{r}{2}\big) & \mbox{if }\;0\leq r\leq\pi,	\\[.7em]
	r-2 						 & \mbox{if }\;r>\pi,	\\
\end{array}\right.$$
is onto and strictly increasing, and $\varphi\in\mathcal{F}$. Let $u,v\in X$.
\begin{enumerate}
\item[\textit{Case 1.}] If $\nor{u-v}_*\leq\pi$.  Notice that for each $t\in[0,1]$ we obtain that
\begin{equation*}
  \begin{split}
  	&\hspace{-3em}\nor{(T-S)(u)-(T-S)(v)}_\infty \\[.5em]
  	& \geq \abs{(T-S)(u)(t)-(T-S)(v)(t)} \\[.5em]
& = \abs{A(u''(t))-A(v''(t))-\big(\sin(u(t))-\sin(v(t))\big)} \\
& \geq \abs{A(u''(t))-A(v''(t))}-2\abs{\sin\big(\tfrac{u(t)-v(t)}{2}\big)\cos\big(\tfrac{u(t)+v(t)}{2}\big)} \\
& \geq \abs{u''(t)-v''(t)}-2\abs{\sin\big(\tfrac{u(t)-v(t)}{2}\big)}\,.
  \end{split}
\end{equation*}
By symmetry we can assume that $u(t) \geq v(t)$, then $\sin\big(\tfrac{u(t)-v(t)}{2}\big)\geq0$. In this case,
\begin{equation*}
  \begin{split}
  	\nor{(T-S)(u)-(T-S)(v)}_\infty & \geq \abs{u''(t)-v''(t)}-2\sin\big(\tfrac{u(t)-v(t)}{2}\big) \\
  	 & \geq \abs{u''(t)-v''(t)}-2\sin\big(\tfrac{\nor{u-v}_*}{2}\big)\,. \\
  \end{split}
\end{equation*}
because the function $r\mapsto\sin(r)$ is increasing in $[0,\tfrac{\pi}{2}]$. Then,
\begin{equation*}
  \begin{split}
  	\nor{(T-S)(u)-(T-S)(v)}_\infty & \geq \nor{u''-v''}_\infty - 2\sin\big(\tfrac{\nor{u-v}_*}{2}\big) \\
  	& = \nor{u-v}_* - 2\sin\big(\tfrac{\nor{u-v}_*}{2}\big) \\
  	& = \varphi(\nor{u - v}_*)\,.
  \end{split}
\end{equation*}

\item[\textit{Case 2.}] Suppose that $\nor{u-v}_*>\pi$. Using~\eqref{eq:Dif.Eq:01}, we get that
\begin{equation*}
  \begin{split}
  	\nor{(T-S)(u)-(T-S)(v)}_\infty & \geq \nor{T(u)-T(v)}_\infty - \nor{\sin(u)-\sin(v)}_\infty \\
  							  	   & \geq\nor{u-v}_*-2=\varphi(\nor{u-v}_*)\,.\\
  \end{split}
\end{equation*}
\end{enumerate}

By Theorem~\ref{thm:Gen.Theorem} there exists a unique $x\in X$ such that $S(x)=T(x)$, that is, $x$ is a unique solution of the problem~($P$). Furthermore, using Proposition~\ref{prop:Ulam.Hyers.stability} we have that problem~($P$) is generalized Ulam-Hyers stable.
\end{proof}
Notice that the assumption~($A_2$) in Proposition~\ref{pro:Dif.Eq} comes natural, because we can easily find functions satisfying~($A_2$). For example, given $k\in\R$ with $k\geq2$, the function $A:\R_+\to\R_+$ defined by
$$
Ax:=\left\{
\begin{array}{ll}
	2\sqrt{x} & \mbox{if }0\leq x\leq1, \\
	& \\
	k\,x	  & \mbox{if }x>1,
\end{array}
\right.
$$
satisfies property~($A_2$) with
$$f(t)=\min\left\{\frac{t^2}{4},\frac{t}{k}\right\}.$$
In order to show this, we shall distinguish three cases. Let $x,y\in\R_+$.
\begin{itemize}
\item[\textit{Case 1.}] Suppose that $0\leq y\leq x\leq 1$. Notice that $\abs{Ax-Ay}= 2\sqrt{x}-2\sqrt{y}\leq\frac{4}{k}$. Then,
\begin{equation*}
\begin{split}
f(\abs{Ax-Ay})&=\left(\sqrt{x}-\sqrt{y}\right)^2\leq x-y\\[.5em]
&\leq\frac{1}{\sqrt{\xi}}\,(x-y)=\abs{A'(\xi)}\,\abs{x-y}= \abs{Ax-Ay},
\end{split}
\end{equation*}
where $\xi\in(x,y)$.
\item[\textit{Case 2.}]If $0\leq y\leq 1<x$. In this case,
$$f(\abs{Ax-Ay})\leq\frac{\abs{Ax-Ay}}{k}\leq x-\sqrt{y}\leq x-y.$$
Moreover,
$$\abs{x-y}\leq\abs{Ax-Ay}\Longleftrightarrow 2\sqrt{y}-y\leq(k-1)x$$
which is true because $0\leq y\leq 1<x$ and $k\geq2$.
\item[\textit{Case 3.}] Assume that $1\leq y\leq x$.
$$f(\abs{Ax-Ay})\leq\frac{\abs{Ax-Ay}}{k}=\abs{x-y}\leq k\,\abs{x-y}=\abs{Ax-Ay}.$$
\end{itemize}
\begin{remark}
Note that property~($A_1$) in Proposition~\ref{pro:Dif.Eq} is necessary, because~($A_2$) does not imply the continuity of $A$. Indeed, just take $k>2$ in the previous example.
\end{remark}

Recently in \cite{GF-Ml.13}, it has been proved the existence and uniqueness of solutions for the following problem
$$(P_a)\left\{
\begin{array}{l}
	u''(t)-a^2\sin\big(u(t)\big)=f_0(t),\qquad\mbox{for }t\in[0,1]	\\[1em]
	u(0)=0,\quad u(1)=0,
\end{array}\right.$$
where the fixed function $f_0\in Y$ is called the driving force, and the constant $a\neq0$ depends on the length of the pendulum and on gravity. In their result, they required the hypothesis $\abs{a}<1$. We can apply the previous proposition in order to improve this result assuming that $\abs{a}\leq1$. Indeed, it is enough to take $A(r)=\frac{r}{a^2}$ and $g(t)=\frac{f_0(t)}{a^2}$.
\begin{remark}
The Ulam-Hyers stability can be used to get a region of localization of the solution of problem~($P$) although this solution becomes unknown. For instance, for the previous problem~($P_a$), with $\abs{a}=1$ and $f_0(t)=sin(\pi t)$ (see~\cite[Section~4.7]{Tri.85})
taking the following functions as initial datum
\begin{equation*}
\begin{split}
w_1(t)&=0\\
w_2(t)&=(t-1)\frac{t}{4}\\
w_3(t)&=-\frac{\sin(\pi\,t)}{\pi^2}\\
w_4(t)&=-\frac{\sin(\pi\,t)}{\pi^2}+\sin\big(\tfrac{\sin(\pi\,t)}{\pi^4}\big)\\
\end{split}
\end{equation*}
we obtain the computable results shown in Table~\ref{table1} and Figure~\ref{figure1} by using \textit{Mathematica$^\circledR$} version 9.0.
\end{remark}
\begin{table}[!ht]
\begin{center}
\begin{tabular}{|c|c|c|}
 \hline
 \rule[-1.5ex]{0pt}{4.5ex} Initial data & $\eps$ & $\psi(\eps)$ \\
 \hline
 \rule[-1.5ex]{0pt}{4.5ex} $w_1(t)$ & 1 & 2.994600778191 \\
 \hline
 \rule[-1.5ex]{0pt}{4.5ex} $w_2(t)$ & 0.5 & 2.342459305003 \\
 \hline
 \rule[-1.5ex]{0pt}{4.5ex} $w_3(t)$ & 0.1011479123607 & 1.354285018462 \\
 \hline
 \rule[-1.5ex]{0pt}{4.5ex} $w_4(t)$ & 0.0103862353036 & 0.630389524267 \\
 \hline
 \end{tabular}
\end{center}
\caption{Results of generalized Ulam-Hyers stability for problem~($P_a$) in the sense that if $\nor{Tw_i-Sw_i}_\infty\leq\eps$ then $\nor{w_i-x}_\infty\leq\psi(\eps)$, where $x$ is the unique solution of~($P_a$).
}\label{table1}
\end{table}
\begin{figure}[!ht]
\begin{center}
\includegraphics[scale=.4]{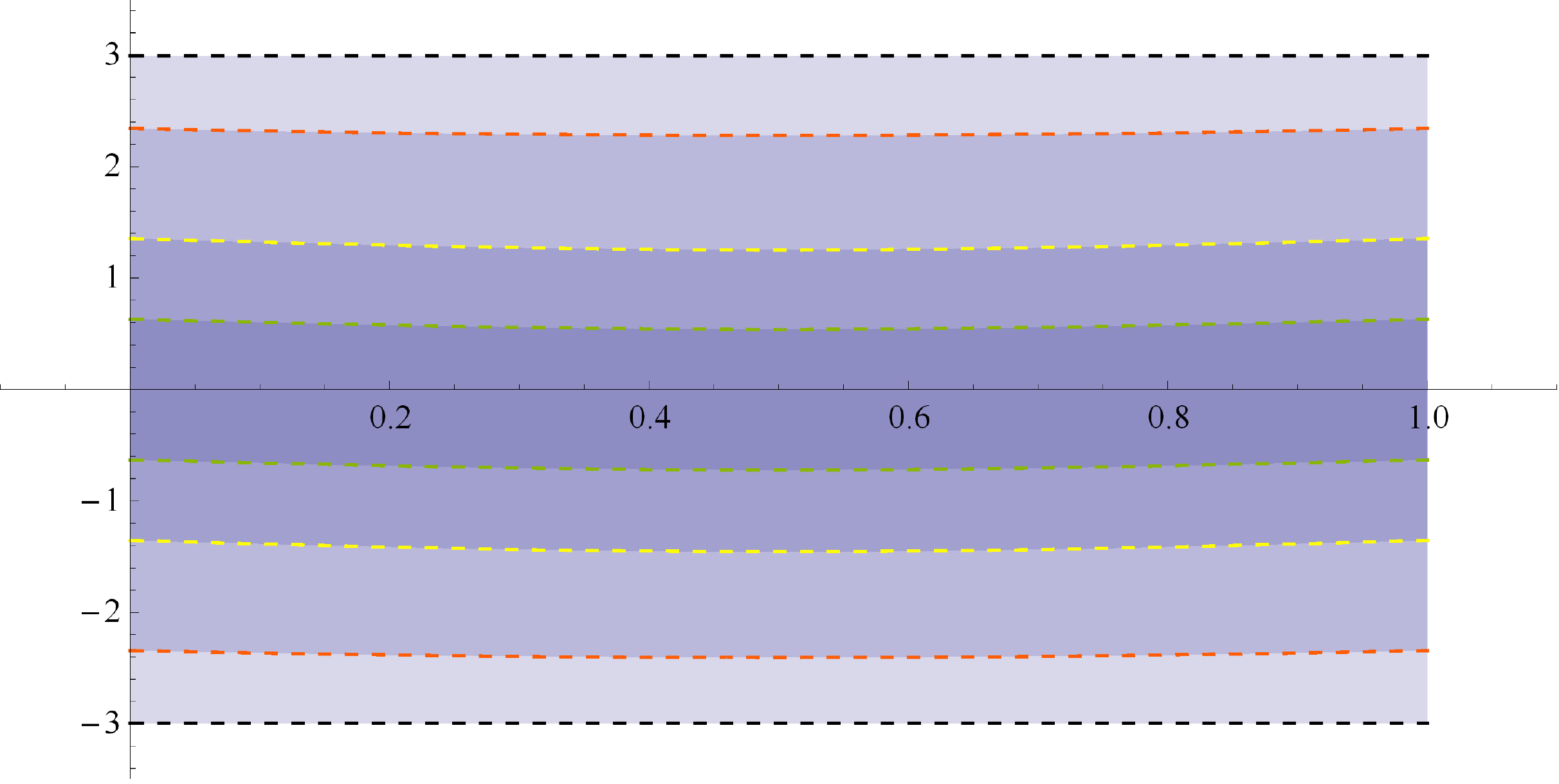}
\caption{Several regions of localization of the solution of problem~($P$) for the above initial datum $w_i$.}\label{figure1}
\end{center}
\end{figure}
%
%====================================================================================================================================================================================================================================================
%
\subsection{An application to Fractional Differential Equations}\label{sec:Fractional.Dif.Eq}
%
%====================================================================================================================================================================================================================================================
%
In this section we shall study the Cauchy problem~\eqref{eq:CP} with the nonlocal conditions for fractional differential equations of Caputo type
\begin{equation*}\label{eq:CP}\tag{CP}
\left\{
\begin{array}{l}
^cD^q x(t)=f(t,x(t))\qquad\mbox{in }\R_+, 	\\[1em]
x(0)=x_0+g(x),
\end{array}
\right.
\end{equation*}
where $f\in\mathcal{C}(\R_+\times\R)$, $0<q<1$, $^cD^qx$ is the fractional derivative of $x$ which is defined by
$$^cD^qx(t):=\frac{1}{\Gamma(1-q)}\int_0^t(t-s)^{-q}x'(s)\,ds,$$
where $\Gamma$ denotes to the Gamma function, $x_0\in\R$, and $g(x)$ is defined by
$$g(x)=\sum_{i=1}^Ng_i(x(t_i)),$$
where each $g_i:\R\to\R$ is $c_i$-lipschitziann, and $0<t_1<t_2<\cdots<t_N<\infty$. For instance, Deng~\cite{De.93} used this class of nonlocal condition with $g_i(x(t_i))=c_ix(t_i)$, for each $i=1,\ldots,N$, pointing out that, unlike the classical Cauchy problem with initial condition $x(0)=x_0$, one can obtain a better effect using the nonlocal condition $x(0)+g(x)=x_0$ in certain physical processes, for instance, in order to describe the diffusion phenomenon of a small amount in a transparent tube. In this case, the Cauchy problem allows the additional measurements at $t_i$, for $i=1,\ldots,N$. Recently, N'Gu\'{e}r\'{e}kata~\cite{NG.09} proved the existence and uniqueness of solutions to problem~(\ref{eq:CP}) on a bounded interval.
\par\medskip
On one hand, notice that $g$ is $L_g$-lipschitzian with $L_g=\sum_{i=1}^Nc_i$, since each $g_i$ is $c_i$-lipschitzian. On the other hand, the initial value problem~\eqref{eq:IVP} for fractional differential equations of Caputo type is a particular case of the Cauchy problem~\eqref{eq:CP}.
\begin{equation*}\label{eq:IVP}\tag{IVP}
\left\{
\begin{array}{l}
^cD^q x(t)=f(t,x(t))\qquad\mbox{in }\R_+, 	\\[1em]
x(0)=x_0.
\end{array}
\right.
\end{equation*}
Indeed, in this case, $g$ is the null function and thus,
for convenience, we may consider that $t_N=0$.
\par\medskip
Since $f$ is assumed continuous, \eqref{eq:CP} is equivalent to the following Volterra  integral equation.
\begin{equation}\label{eq:Volterra}
x(t)=x_0+g(x)+\frac{1}{\Gamma(q)}\int_0^t(t-s)^{q-1}f\big(s,x(s)\big)\,ds,
\qquad\mbox{for }t\geq0.
\end{equation}

Then, every solution of~\eqref{eq:Volterra} is also a solution of~\eqref{eq:CP} and vice versa. See~\cite[p.54]{La-Le-Va.09}, or~\cite[pp.78,86,103]{Di.04}, for its proof.

\begin{theorem}\label{thm:Existence.Uniqueness.CP}
Let $0<q<1$. Assume that $f:\R_+\times\R\to\R$ is a continuous function. If there exists a positive constant $L_f$ such that
$$\abs{f(s,u)-f(s,v)}\leq L_f\abs{u-v}\qquad\mbox{for all }u,v\in\R\mbox{ and a.e. }s\geq0,$$
then equation~\eqref{eq:Volterra} (and, therefore~\eqref{eq:CP}) has a unique solution in $\mathcal{C}\R_+$ whenever
$$
\frac{L_f}{\Gamma(q)}\left(\frac{t_N^q}{q}\right)+L_g<1.
$$
\end{theorem}

In order to prove the uniqueness of solutions of problem~\eqref{eq:CP} we need the following result which is a consequence of Bellman's inequality.

\begin{lemma}\label{lemma:u.0}
Let $u:[0,T]\to\R_+$ be a continuous function such that
$$
u(t)\leq\alpha\,\int_0^t(t-s)^{q-1}u(s)\,ds\qquad\mbox{for all }t\in[0,T],
$$
where $\alpha,q>0$. Then, $u(t)=0$ for all $t\in[0,T]$
\end{lemma}

\begin{proof}
Consider $p=\frac{q+1}{q}>1$ and $\hat{p}=q+1$. Note that $\frac{1}{p}+\frac{1}{\hat{p}}=1$. By H\"{o}lder's inequality, for each $t\in(0,T]$,
\begin{equation*}
\begin{split}
u(t)&\leq \int_0^t\alpha\,(t-s)^{q-1}e^s\,e^{-s}u(s)\,ds\\
& \leq\alpha\left(\int_0^t(t-s)^{(q-1)\hat{p}}e^{\hat{p}s}\,ds\right)^{\frac{1}{\hat{p}}}\left(\int_0^t\big(e^{-s}u(s)\big)^p\,ds\right)^{\frac{1}{p}}\\
&<\alpha\,M(q)\,e^t\,\left(\int_0^t\big(e^{-s}u(s)\big)^p\,ds\right)^{\frac{1}{p}},
\end{split}
\end{equation*}
where $M(q):=\big(\Gamma(q^2)\,(q+1)^{-q^2}\big)^{\frac{1}{q+1}}$, because
\begin{equation*}
\begin{split}
\int_0^t(t-s)^{(q-1)\hat{p}}e^{\hat{p}s}\,ds
&=e^{\hat{p}t}\int_0^t r^{(q-1)\hat{p}}e^{-\hat{p}r}\,dr
=\frac{e^{\hat{p}t}}{\hat{p}}\int_0^{\hat{p}t}\big(\tfrac{z}{\hat{p}}\big)^{(q-1)\hat{p}}e^{-z}\,dz\\
&=\frac{e^{\hat{p}t}}{\hat{p}^{(q-1)\hat{p}+1}}\int_0^{\hat{p}t} z^{(q-1)\hat{p}}e^{-z}\,dz
 =e^{\hat{p}t}\,\hat{p}^{-q^2}\int_0^{\hat{p}t} z^{q^2-1}e^{-z}\,dz \\
&<e^{\hat{p}t}\,\hat{p}^{-q^2}\,\Gamma\big(q^2\big),
\end{split}
\end{equation*}
since $(q-1)\hat{p}+1=q^2$. Then,
$$
\big(e^{-t}u(t)\big)^p\leq\alpha^p\,M(q)^p\,\int_0^t\big(e^{-s}u(s)\big)^p\,ds\qquad\mbox{for all }t\in[0,T].
$$
By Bellman's inequality,
$$
\big(e^{-t}u(t)\big)^p\leq 0\cdot \mbox{Exp}\big(\alpha^p\,M(q)^p\,t\big)=0\qquad\mbox{for each }t\in[0,T].
$$
Therefore, $u(t)=0$ for all $t\in[0,T]$.
\end{proof}
Now, we can give the proof of Theorem~\ref{thm:Existence.Uniqueness.CP}.
\begin{proof}
Since $\lim_{\lambda\to\infty}(\lambda L_f)^q=+\infty$. We may take a $\lambda>0$  satisfying both $\dfrac{q}{L_f\,t_N}<\lambda$ and $\frac{L_f}{\Gamma(q)}\left(\frac{t_N^q}{q}+\frac{\Gamma(q)}{\;(\lambda L_f)^q}\right)+L_g<1$. Let us consider the set
$$X_\lambda:=\Big\{u\in\mathcal{C}(\R_+):\sup_{t\geq0}\frac{\abs{u(t)}}{\omega_{\lambda}(t)}<\infty\Big\},$$
where $\omega_{\lambda}:\R_+\to[1,\infty)$ is defined by
$$
\omega_\lambda(t):=\left\{
\begin{array}{ll}
e^{\lambda L_f t_N} & \mbox{if }0\leq t\leq t_N,	\\
&	\\
e^{\lambda L_f t} & \mbox{if }t\geq t_N,	\\
\end{array}\right.
$$
and the Banach space $Y=\mathcal{BC}(\R_+)$ of bounded continuous functions endowed with the norm
$\nor{u}_\infty:=\sup_{t\geq0}\abs{u(t)}$. We define the mappings $T,S:X_\lambda\to Y$ given by
$$
T(u)(t)=\frac{u(t)-x_0}{\omega_\lambda(t)}
$$
and
$$
S(u)(t)=\frac{1}{\omega_\lambda(t)\,\Gamma(q)}\int_0^t(t-s)^{q-1}f(s,u(s))\,ds+\frac{g(u)}{\omega_\lambda(t)}.
$$
Notice that~\eqref{eq:Volterra} can be written as a coincidence problem in the following form:
$$\mbox{find $u\in X_\lambda$ such that $T(u)=S(u)$.}$$
Let $u,v\in X_\lambda$. For each $t\geq0$, we have that
\begin{equation*}
\begin{split}
\abs{S(u)(t)-S(v)(t)}
&\leq\frac{1}{\omega_\lambda(t)\,\Gamma(q)}\int_0^t(t-s)^{q-1}\abs{f(s,u(s))-f(s,v(s))}\,ds \\[.5em]
&\quad+\frac{\abs{g(u)-g(v)}}{\omega_\lambda(t)}	\\
&\leq\frac{1}{\omega_\lambda(t)\,\Gamma(q)}\int_0^t(t-s)^{q-1}L_f\,\abs{u(s)-v(s)}\,ds \\[.5em]
&\quad+\frac{\sum_{i=1}^Nc_i\,\abs{u(t_i)-v(t_i)}}{\omega_\lambda(t)}	\\
&\leq\frac{L_f}{\omega_\lambda(t)\,\Gamma(q)}\int_0^t(t-s)^{q-1}\abs{u(s)-v(s)}\,ds \\[.5em]
&\quad+\sum_{i=1}^N c_i\frac{\abs{u(t_i)-v(t_i)}}{\omega_\lambda(t_i)},
\end{split}
\end{equation*}
because $\omega_\lambda(t)\geq\omega_\lambda(t_i)$ for all $t\geq0$. Hence,
\begin{equation*}
\begin{split}
\abs{S(u)(t)-S(v)(t)}&
\leq\frac{L_f}{\Gamma(q)}\int_0^t(t-s)^{q-1}\frac{\omega_\lambda(s)}{\omega_\lambda(t)}\,ds\,\nor{T(u)-T(v)}_\infty \\[.5em]
&\quad+L_g\,\nor{T(u)-T(v)}_\infty\,.
\end{split}
\end{equation*}
Now we shall prove that, for every $t\geq0$,
$$\int_0^t(t-s)^{q-1}\frac{\omega_\lambda(s)}{\omega_\lambda(t)} ds\leq
\frac{t_N^q}{q}+\frac{\Gamma(q)}{\;(\lambda L_f)^q}\,.
$$
In order to do this, we consider two cases.
\begin{itemize}
\item[\textit{Case 1.}] If $0\leq t\leq t_N$. Then, $\dfrac{\omega_\lambda(s)}{\omega_\lambda(t)}=1$ for all $0\leq s\leq t$. Hence,
$$
\int_0^t(t-s)^{q-1}\frac{\omega_\lambda(s)}{\omega_\lambda(t)} ds=
\int_0^t(t-s)^{q-1} ds=\frac{t^q}{q}\leq\frac{t_N^q}{q}.
$$
\item[\textit{Case 2.}] If $t\geq t_N$,
\begin{equation*}
\begin{split}
\int_0^t(t-s)^{q-1}&\frac{\omega_\lambda(s)}{\omega_\lambda(t)} ds\\
  &=\int_0^{t_N}(t-s)^{q-1}\frac{\omega_\lambda(s)}{\omega_\lambda(t)} ds+
	\int_{t_N}^t(t-s)^{q-1}\frac{\omega_\lambda(s)}{\omega_\lambda(t)} ds	\\
&=\int_0^{t_N}(t-s)^{q-1}e^{-\lambda L_f(t-t_N)} ds+
  \int_{t_N}^t(t-s)^{q-1}e^{-\lambda L_f(t-s)} ds	\\
&\leq
  e^{-\lambda L_f(t-t_N)} \frac{t^q-(t-t_N)^q}{q}+
  \frac{\Gamma(q)}{\;(\lambda L_f)^q}	\\
\end{split}
\end{equation*}
because
\begin{equation*}
\begin{split}
\int_{t_N}^t(t-s)^{q-1}e^{-\lambda L_f(t-s)} ds&=
\int_0^{t-t_N}x^{q-1}e^{-\lambda L_f x} dx\\
&\leq
\int_0^\infty x^{q-1}e^{-\lambda L_f x} dx=
\frac{\Gamma(q)}{\;(\lambda L_f)^q}\,.
\end{split}
\end{equation*}
Moreover, for each $t\geq t_N$ we have that
\begin{equation*}
 \begin{split}
   \frac{t^q-(t-t_N)^q}{q}\, e^{-\lambda L_f(t-t_N)} &\leq
   \frac{t^q}{q}\, e^{-\lambda L_f(t-t_N)} \leq
   \frac{1}{q}\,\sup_{s\geq t_N} s^q\,e^{-\lambda L_f(s-t_N)}
   =\frac{t_N^q}{q}
 \end{split}
\end{equation*}
because $\dfrac{q}{L_f\,t_N}<\lambda$. Then,
$$
\int_0^t(t-s)^{q-1}\frac{\omega_\lambda(s)}{\omega_\lambda(t)} ds
\leq\frac{t_N^q}{q}+\frac{\Gamma(q)}{\;(\lambda L_f)^q}\,.
$$
\end{itemize}
Therefore, for all $t\geq0$,
$$
\abs{S(u)(t)-S(v)(t)}\leq
\left[\frac{L_f}{\Gamma(q)}\left(\frac{t_N^q}{q}+\frac{\Gamma(q)}{\;(\lambda L_f)^q}\right)+L_g\right]
\nor{T(u)-T(v)}_\infty\,.
$$
Then,
$$\nor{S(u)-S(v)}_\infty\leq
\left[\frac{L_f}{\Gamma(q)}\left(\frac{t_N^q}{q}+\frac{\Gamma(q)}{\;(\lambda L_f)^q}\right)+L_g\right]
\,\nor{T(u)-T(v)}_\infty.$$
Applying Corollary~\ref{cor:Coincidence.Contraction}, we deduce that there exists a unique $u\in X_\lambda$ such that  $S(u)=T(u)$, that is, problem~\eqref{eq:IVP} has at least one solution in $X_\lambda$, in particular, in $\mathcal{C}(\R_+)$.
\par\medskip
Now, we prove the uniqueness of solutions of problem~\eqref{eq:CP} in $\mathcal{C}(\R_+)$. Assume that there exist two solutions $x$ and $y$ of problem~\eqref{eq:CP} in $\mathcal{C}(\R_+)$. Let us see that $x(t)=y(t)$ for all $t\geq0$. Notice that, for each $t\geq0$,
\begin{equation*}
\begin{split}
\abs{x(t)-y(t)}&\leq \abs{g(x)-g(y)}+\int_0^t (t-s)^{q-1}\,\abs{f(s,x(s))-f(s,y(s))}\,ds\\
&\leq\sum_{i=1}^N c_i\abs{x(t_i)-y(t_i)}+\int_0^t (t-s)^{q-1}\,L_f\,\abs{x(s)-y(s)}\,ds.\\
\end{split}
\end{equation*}
Let $t_*\in\{t_1,\ldots,t_N\}$ be such that
$$\abs{x(t_*)-y(t_*)}=\max_{1\leq i\leq N}\abs{x(t_i)-y(t_i)}.$$
Then, for $t=t_*$, we have that
$$
\abs{x(t_*)-y(t_*)}\leq L_g\,\abs{x(t_*)-y(t_*)}+\int_0^{t_*} (t_*-s)^{q-1}\,L_f\,\abs{x(s)-y(s)}\,ds.
$$
Since $0\leq L_g<1$, there exists $\rho\in(L_g,1)$. Hence,
$$
\abs{x(t_*)-y(t_*)}<\int_0^{t_*} k(s)\,\abs{x(s)-y(s)}\,ds,
$$
where $k(s)=(t_*-s)^{q-1}\,\dfrac{L_f}{1-\rho}$. By continuity, there exists $\delta>0$ such that for all $t\in[t_*-\delta,t_*+\delta]$
$$
\abs{x(t)-y(t)}\leq\int_0^{t} k(s)\,\abs{x(s)-y(s)}\,ds.
$$
Using Bellman's inequality we deduce that for all $t\in[t_*-\delta,t_*+\delta]$
$$
\abs{x(t)-y(t)}\leq0\,e^{\int_0^{t} k(s)\,ds}=0.
$$
In particular, for $t=t_*$ we obtain $x(t_*)=y(t_*)$. By definition of $t_*$, we get that $x(t_i)=y(t_i)$ for all $i=1,\ldots,N$. Then, for all $t\geq0$,
$$
\abs{x(t)-y(t)}\leq\int_0^t (t-s)^{q-1}\,L_f\,\abs{x(s)-y(s)}\,ds.
$$
Fix $T>0$. By Lemma~\ref{lemma:u.0}, we have that $x(t)=y(t)$ for all $t\in[0,T]$. Taking $T\to\infty$, we deduce that $x=y$.
\end{proof}
As a consequence of Theorem~\ref{thm:Existence.Uniqueness.CP}, we get the existence and uniqueness of solutions of problem~\eqref{eq:IVP} in $\mathcal{C}(\R_+)$.
\begin{corollary}
If $f:\R_+\times\R\to\R$ is a continuous function and lipschitzian with respect to the second variable, then
 problem~\eqref{eq:IVP} has a unique solution in $\mathcal{C}(\R_+)$.
\end{corollary}
%
%====================================================================================================================================================================================================================================================
%

\section*{Acknowledgements} The research of the first author has been partially supported by MTM 2012-34847-C02-01 and P08-FQM-03453. The second author has been partially supported by MTM 2012-34847-C02-02.


% \section*{References}

%% References

 %%References with bibTeX database:

\bibliographystyle{elsarticle-num}

%%\bibliography{<your-bib-database>}






% \end{thebibliography}

\end{document}